\documentclass[a4paper,11pt]{article}

\usepackage{float}

\usepackage{amsmath,amsthm} 
\usepackage{amssymb}
\usepackage{amsfonts}
\usepackage{epstopdf}
\usepackage{color,fancybox,graphicx}
\usepackage{url}
\usepackage{psfrag}
\usepackage{color}
\usepackage{pifont}
\usepackage{graphicx,rotating,booktabs}

\usepackage[small,bf]{caption}
\definecolor{newgreen}{rgb}{0.0, 0.5, 0.3}
\definecolor{newblue}{rgb}{0.0, 0.1, 0.7}

\usepackage{natbib}
\bibpunct{(}{)}{;}{a}{,}{,}
\usepackage{hyperref}
\hypersetup{
  colorlinks = true,
  urlcolor = blue, 
  linkcolor = blue,
  citecolor = blue,
}

\numberwithin{equation}{section}

\newtheorem{rem}{Remark}[section]

\newtheorem{cor}{Corollary}[section]
\newtheorem{pro}{Proposition}[section]
\newtheorem{theo}{Theorem}[section]

\newcommand{\R}{\mathbb R}
\newcommand{\prob}{\mathbb{P}}
\newcommand{\esp}{\mathbb{E}}

\newcommand{\ra}{\rightarrow}
\newcommand{\e}{\varepsilon}

\newcommand{\bs}{\boldsymbol}

\begin{document}

\begin{center}
{\Large\sc Finite sample properties of the mean occupancy counts and probabilities}\vspace{0.2cm}\\
\end{center}
\begin{center}
{\bf G. Decrouez$^{\dagger}$, M. Grabchak$^{\star}$ and Q. Paris$^{\dagger,\sharp,}$\footnote{Corresponding author.}}\vspace{0.1cm}\\
{\it $^{\dagger}$National Research University, Higher School of Economics\footnote{The study has been funded by the Russian Academic Excellence Project 5-100.}\\ 
\& $^{\sharp}$Laboratory of Stochastic Analysis and its Applications\\ 
Moscow, Russian Federation\\
\& $^{\star}$Department of Mathematics and Statistics\\
University of North Carolina, Charlotte}\vspace{0.1cm}\\
\href{mailto:ggdecrouez@hse.ru}{ggdecrouez@hse.ru}, \href{mailto:mgrabcha@uncc.edu}{mgrabcha@uncc.edu} and 
\href{mailto:qparis@hse.ru}{qparis@hse.ru}
\end{center}
\vspace{0.2cm}
\begin{quote}{\small
\noindent{\bf Abstract }-- For a probability distribution $P$ on an at most countable alphabet $\mathcal A$, this article gives finite sample bounds for the expected occupancy counts $\esp K_{n,r}$ and probabilities $\esp M_{n,r}$. Both upper and lower bounds are given in terms of the counting function $\nu$ of $P$. Special attention is given to the case where $\nu$ is bounded by a regularly varying function. In this case, it is shown that our general results lead to an optimal-rate control of the expected occupancy counts and probabilities with explicit constants. Our results are also put in perspective with Turing's formula and recent concentration bounds to deduce bounds in probability. At the end of the paper, we discuss an extension of the occupancy problem to arbitrary distributions in a metric space.\vspace{0.2cm}

}
\end{quote}

\date{\today}

%-----------------------------------------
%-----------------------------------------
\section{Introduction}
\label{intro}
{\it The occupancy problem}\vspace{0.3cm}\\
From a general point of view, the occupancy problem -- also referred to as the urn scheme -- is to  describe the spread of a random sample drawn from a probability distribution supported by an at most countable alphabet. In the literature, this task is usually carried out by studying the so-called occupancy counts and occupancy probabilities -- also known as rare probabilities -- defined below. Interest for the occupancy problem arises in many practical situations such as Ecology \citep{GT56, C81}, Genomics \citep{ML02}, Language Processing \citep{CG99}, Authorship Attribution \citep{ET76, TE87,ZH07}, Information Theory \citep{OSZ04} and Computer Science \citep{Z05}.\\ 

Consider an at most countable alphabet $\mathcal A$ with an associated probability distribution $P=\{p_a:a\in\mathcal A\}$, where $p_a\in[0,1]$ and $\sum_{a\in\mathcal A}p_a=1$. Let $\mathcal S=\{a\in\mathcal A:p_a>0\}$ denote the support of $P$, and let $X_1,\dots,X_n$ be independent and identically distributed $\mathcal A$-valued random variables, defined on some probability space $(\Omega, \mathcal F, \mathbb P)$, with distribution $P$. For all $a\in\mathcal A$, we set  
\begin{equation}
\label{xi}
\xi_{n}(a)=\sum_{i=1}^{n}\mathbf 1\{X_{i}=a\},
\end{equation}
where the notation $\mathbf 1\{...\}$ stands for the indicator function of the event $\{...\}$. For all integers $0\le r\le n$, the occupancy counts $K_{n,r}$ and occupancy probabilities $M_{n,r}$ are defined, respectively, by 
\begin{equation}
\label{defKM}
K_{n,r}=\sum_{a\in\mathcal A}\mathbf 1\{\xi_{n}(a)=r\}\quad\mbox{and}\quad M_{n,r}=\sum_{a\in\mathcal A}p_a\mathbf1\{\xi_{n}(a)=r\}.
\end{equation}
For any integer $0\le r\le n$, the random variable $K_{n,r}$ stands for the number of points in $\mathcal A$ represented exactly $r$ times in the sample. A clear interpretation of the occupancy probabilities is given by the following equivalent representation. Introducing a generic $\mathcal A$-valued random variable $X$, independent of the sample and distributed according to $P$, we have, almost surely,
$$M_{n,r}=\prob\left(\xi_n(X)=r\,\vert\,X_1,\dots,X_n\right).$$
Hence, for any integer $0\le r\le n$, $M_{n,r}$ stands for the (conditional) probability that, given the first $n$ observations, the next one will be of a letter that is already represented $r$ times in the sample. The quantity $M_{n,0}$ is particularly important. In the literature it is usually called the missing mass, and has attracted a lot of attention due to its practical interpretation as the probability of novelty.  The goal of this paper is to understand the finite sample properties of $\esp K_{n,r}$ and $\esp M_{n,r}$.\\

{\it Related work}\vspace{0.3cm}\\
Following the pioneering work of \citet{K67}, it is understood that the asymptotic behavior of the occupancy counts $K_{n,r}$ is strongly connected to the behavior of the tail of the counting measure $\bs\nu$ of $P$, which is defined on $[0,1]$ by
\begin{equation}\label{measurenu}
\bs\nu({\rm d}x)=\sum_{a\in\mathcal A}\delta_{p_a}({\rm d}x).
\end{equation}
The function $\nu:[0,1]\to\mathbb N$, defined  by
\begin{equation}
\label{nu}
\nu(\e)=\bs\nu([\e,1]),
\end{equation}
is usually referred to as the counting function of $P$. A short account of some of its basic properties is given in Appendix \ref{nuprop}. We now illustrate the relationship between the behavior of $\nu$ and that of $K_{n,r}$. Toward this end, we recall some terminology from \cite{K67}. We say that a function $f:[0,+\infty)\to\R$ is regularly varying at $x_0\in\{0,\infty\}$ with exponent $\alpha\in\R$, and we write $f\in{\rm rv}^{\alpha}_{x_0}$, if 
$$
\forall c>0, \quad\lim_{x{\to} x_0}\frac{f(cx)}{f(x)} = c^\alpha.
$$
If $\alpha=0$ we say that $f$ is slowly varying at $x_0$. 
Note that $f\in {\rm rv}^{\alpha}_{0}$ if and only if there exists $\ell\in{\rm rv}^{0}_{\infty}$ such that, for all $\e>0$,
\begin{equation}
\label{rvalpha}
f(\e)= \e^{-\alpha}\ell(1/\e).
\end{equation}
It is well-known that the counting function $\nu$, defined in \eqref{nu}, satisfies $\nu(\e)= \e^{-\alpha}\ell(1/\e)$, for some $\alpha\in(0,1)$ and some $\ell\in{\rm rv}^{0}_{\infty}$, if and only if 
\begin{equation}
\forall r\ge 1:\quad K_{n,r}\underset{a.s.}{\sim}\esp K_{n,r}\sim\frac{\alpha\Gamma(r-\alpha)}{r!}n^{\alpha}\ell(n),
\label{asympK}
\end{equation}
as $n\to+\infty$. (Here and throughout, for any two real-valued functions $g$ and $h$ and any $x_{0}\in[0,+\infty]$, we write $h(x)\sim g(x)$ as $x\to x_{0}$ if and only if $h(x)/g(x)\to 1$ as $x\to x_{0}$.) For a detailed exposition and developments on this topic, we refer the reader to the classic text by \citet{Johnson:Kotz:1977} or the more recent, and very complete, survey by \citet{GHP07}, which, in particular, studies extensions of \eqref{asympK} to the case $\alpha\in\{0,1\}$ under additional care. \\

In the same spirit, \citet{OD12} extended \eqref{asympK} to the case of occupancy probabilities proving that, if $\nu(\e)= \e^{-\alpha}\ell(1/\e)$ for some $\alpha\in(0,1)$ and some $\ell\in {\rm rv}^{0}_{\infty}$, then
\begin{equation}
\forall r\ge 0:\quad M_{n,r}\underset{a.s.}{\sim}\esp M_{n,r}\sim\frac{\alpha\Gamma(1+r-\alpha)}{r!} n^{\alpha-1}\ell(n),
\label{asympM}
\end{equation}
as $n\to+\infty$. While the second asymptotic equivalence in \eqref{asympM} is, as mentioned by the authors, easily derived from \eqref{asympK} and the relation
\begin{equation}
\label{KandM}
\esp M_{n,r}=\left(\frac{1+r}{1+n}\right)\esp K_{n+1,r+1},
\end{equation}
the first asymptotic equivalence in \eqref{asympM} is established by \citet{OD12} by proving more powerful concentration properties of $M_{n,r}$ around its expectation.\\

Some of the first concentration properties in this context were established by \citet{MS00} for the missing mass $M_{n,0}$. The concentration properties of the missing mass have since been investigated by \citet{MO03,OD10,BK13} and \citet{KH15}. Many extensions and new results concerning the concentration properties of the occupancy counts $K_{n,r}$ and occupancy probabilities $M_{n,r}$ can be found in \citet{OD12} and \citet{BBO15}.\\

Establishing concentration properties is a fundamental step toward understanding the finite sample behavior of the occupancy counts and probabilities. However, a full understanding of the finite sample properties of $K_{n,r}$ and $M_{n,r}$ requires finite sample bounds for their expectations. We are only aware of two contributions in this direction, namely \citet{OD10} and \citet{BK12}, which both focus on the missing mass. In particular, \citet{OD10} introduce the accrual function $F(x)=P(\{a: p_a\le x\})$, and show that
\begin{equation}
\label{od10}
\sup_{0\le\e\le1}\left\{(1-\e)^nF(\e)\right\}\le \esp M_{n,0}\le\inf_{0\le\e\le1}\left\{(1-\e)^n+F(\e)\right\}.
\end{equation}
It should be noted that, as described in Appendix \ref{appendixod10}, this result yields, in many cases, explicit bounds with almost optimal rates of convergence. In \citet{BK12}, the authors show that, in the finite support case,
$$\forall n\le \vert\mathcal S\vert:\ \esp M_{n,0}\le e^{-n/\vert\mathcal S\vert}\quad\mbox{and}\quad\forall n>\vert\mathcal S\vert:\ \esp M_{n,0}\le \frac{\vert\mathcal S\vert}{ne},$$
while in the infinite support case, there exists a universal constant $c>0$ such that
$$\esp M_{n,0}\le\frac{L(P)}{cn},\quad\mbox{where}\quad L(P)=\sup_{0<\e<1}\{\nu(\e/2)-\nu(\e)\}.$$
In addition, the authors prove that, for any integer $a>1$, there exists a distribution $P$ for which $L(P)=a$ and $\esp M_{n,0}\ge c'a/n$, where $c'>0$ denotes a universal constant. Hence, their bound is shown to be sharp for a certain class of probability distributions. Unfortunately, $L(P)=+\infty$ in many interesting cases, including when $\mathcal A=\{1,2,\dots\}$ and, for some $\alpha\in(0,1)$, the distribution $P$ has masses $p_k=Ck^{-1/\alpha},\ k \in\mathcal A$. \\

Concerning lower bounds, there are interesting results from a somewhat different perspective given in Lemma 4.1 of \cite{ABS2000} and Lemma 1 of \cite{Zhang2016}. These are discussed, in detail, in Appendix \ref{AppC}. \\

{\it Contribution and organisation of the paper}\vspace{0.3cm}\\
Building on the previous work from \citet{OD10} and \citet{BK12}, this paper establishes finite sample upper and lower bounds for the expected occupancy counts $\esp K_{n,r}$ and the expected occupancy probabilities $\esp M_{n,r}$ for arbitrary $n\ge 1$ and arbitrary $0\le r\le n$. For simplicity of exposition, focus is put on the expected occupancy probabilities $\esp M_{n,r}$ knowing that relation \eqref{KandM} immediately implies similar bounds for the expected occupancy counts. Section \ref{sg} is devoted to our main results. We first give general bounds in terms of the counting function $\nu$, which make no assumptions about the underlying distribution. Additional assumptions on $\nu$ are used to derive more explicit bounds. In particular, when the counting function is regularly varying, the bounds are shown to be consistent with \eqref{asympK} and \eqref{asympM}, and are thus rate optimal. Section \ref{apps} presents some applications and extensions.  Specifically,  Subsection \ref{sstf} discusses the relationship between our results and Turing's formula, while  Subsection \ref{sscr} shows how we can combine our results with recent concentration bounds to derive bounds in probability for $M_{n,r}$ and $K_{n,r}$. Further, in Subsection \ref{ssp}, we present an extension to the case of a random number of observations modelled by a non-homogeneous Poisson process, and in Subsection \ref{ssap} we  discuss an interesting perspective for future research in the context of arbitrary probability measures -- i.e.\ not necessarily discrete -- in a metric space. Proofs are postponed to Section \ref{proofs}. Finally, Appendix \ref{nuprop} collects a few basic properties of the counting function, Appendix \ref{appendixod10} investigates the performance of bounds given in terms of the accrual function, and Appendix \ref{AppC} discusses the lower bounds from \cite{ABS2000} and \cite{Zhang2016}.\\

{\it Notation}\vspace{0.3cm}\\
Throughout, the notation $\mathbf 1\{...\}$ stands for the indicator function of the event $\{...\}$. For any set $B$, we write $\vert B\vert$ to denote the number (possibly infinite) of elements in $B$. For any $t>0$ and any $x\ge 0$, we denote by
\begin{equation}
\label{lig}
\gamma(t,x)=\int_0^{x}u^{t-1}e^{-u}{\rm d}u
\end{equation}
the lower incomplete Gamma function. Note that the Gamma function is given by $\Gamma(t)=\gamma(t,+\infty)$. 

%-----------------------------------------
%-----------------------------------------

\section{Main results}
\label{sg}

 In this section we give upper and lower bounds for the expected occupancy probabilities $\esp M_{n,r}$. From \eqref{KandM} it follows that all of the results in this section can be immediately adapted to the expected occupancy counts $\esp K_{n,r}$. However, for ease of exposition, we only report results in terms of the occupancy probabilities.

\subsection{Upper bounds}

Let $P=\{p_a:a\in\mathcal A\}$ be a probability measure on the countable alphabet $\mathcal A$. Its counting function $\nu$ defined in \eqref{nu}, can be equivalently written as 
\begin{equation}
{\nu}(\e)=\vert\left\{a\in\mathcal A:p_a\ge\e\right\}\vert, \quad 0\le\e\le1.
\label{nu2}
\end{equation}
A short account of the basic properties of $\nu$ is given in Appendix \ref{nuprop}. Our first result provides a general upper bound in terms of $\nu$. In the sequel, we denote 
\begin{equation}
\label{cr}
c(r)=\bigg\{\begin{array}{ll}
e^{-1} &\mbox{ if }\ r=0,\\
e(1+r)/\sqrt{\pi} &\mbox{ if }\ r\ge 1.
\end{array}
\end{equation}

\begin{theo}
\label{tg}
For any $n\ge 1$ and any $0\le r\le n-1$, we have
\begin{equation}
\label{eq:main}
\esp M_{n,r}\le \inf_{0\le\e\le1}\left\{\varphi^{\,+}_{n,r}(\e)+\psi^{\,+}_{n,r}(\e)\right\},
\end{equation}
where
\begin{eqnarray}
\varphi^{\,+}_{n,r}(\e)&=&\frac{c(r)\nu(\e)}{n},
\nonumber\\
\psi^{\,+}_{n,r}(\e)&=&2^{1+r}\binom{n}{r}\int_{0}^{\e}\nu\left(\frac u2\right)u^r\left(1-\frac u2\right)^{n-r}{\rm d}u.
\nonumber
\end{eqnarray}
Further, for any $n\ge 1$,
\begin{equation}
\label{requaln}
\esp M_{n,n}\le \inf_{0\le\e\le1}\left\{p^{n+1}_{\star} \nu(\e)+\e^n\right\},
\end{equation}
where $p_{\star} = \max\{p_a:a\in\mathcal A\}\in(0,1]$. 
\end{theo}

\begin{rem}
The proof of Theorem \ref{tg} consists in studying separately, and for all $\e\in[0,1]$, the contributions of large (i.e. larger then $\e$) and small (i.e. smaller then $\e$) probabilities. These contributions are bounded, respectively, by $\varphi^{\,+}_{n,r}(\e)$ and $\psi^{\,+}_{n,r}(\e)$. Details in the proof reveal that the term $\psi^{\,+}_{n,r}(\e)$ can in fact be replaced by the quantity 
$$
\frac{b^{1+r}}{b-1}\binom{n}{r}\int_{0}^{\e}\nu\left( \frac{u}{b}\right)u^r\left(1-\frac{u}{b}\right)^{n-r}{\rm d}u,
$$
for any $b>1$. In principle, the value of $b$ may be optimized, but for the sake of simplicity, we choose $b=2$. Note that, since $\nu$ is bounded on intervals away from $0$, this should not affect the bound in a substantial way.
\end{rem}

Observe that, in \eqref{eq:main}, the two terms $\varphi^{\,+}_{n,r}(\e)$ and $\psi^{\,+}_{n,r}(\e)$ have opposite monotonic behaviours in $\e$. In full generality, the value of $\e$ leading to the optimal tradeoff is not obvious. However, in many interesting cases, a relevant choice of $\e$ yields explicit and, as far as we know, new bounds.

\begin{cor}
\label{ctgfinite}
Suppose that $\mathcal S$ is finite. Then, for all $n\ge 1$ and all $0\le r\le n-1$,
$$\esp M_{n,r}\le \frac{c(r)\vert\mathcal S\vert}{n}\qquad\mbox{and}\qquad \esp M_{n,n}\le p^{n+1}_{\star}\vert\mathcal S\vert,
$$
where $c(r)$ is as in \eqref{cr}.
\end{cor}

The proof of Corollary \ref{ctgfinite} simply involves taking $\e=0$ in Theorem \ref{tg} and is therefore omitted. Note that, when we take $r=0$ in Corollary \ref{ctgfinite}, we recover the bound $\esp M_{n,0}\le\vert\mathcal S\vert/(ne)$ for the expected missing mass provided by \citet{BK12}. Next, we study several situations, where $P$ has an infinite support.

\begin{cor}
\label{ctg}
Suppose that $\mathcal S$ is infinite. Assume that, for $\alpha\in[0,1]$ and $\ell\in {\rm rv}^0_{\infty}$, we have $\nu(\e)\le \e^{-\alpha}\ell(1/\e)$ for all $0<\e\le1$. Suppose, in addition, that $\ell$ is non-increasing. Then, for all $n\ge 2$ and all $0\le r\le n-1$, we have  
$$\esp M_{n,r}\le c_{1}(\alpha,r)\,n^{\alpha-1}\ell(n),$$
where 
$$c_{1}(\alpha,r)= c(r)+\frac{4^{1+r}}{r!}(1+r)^{1+r-\alpha}\gamma(1+r-\alpha,\tfrac 12),$$
$c(r)$ is as in \eqref{cr}, and $\gamma(\cdot,\cdot)$ denotes the lower incomplete Gamma function defined in \eqref{lig}.
\end{cor}
According to \eqref{asympM}, the bound of Corollary \ref{ctg} is rate optimal in terms of $n$. In order for the bound to be even more explicit, note that for all $t>0$ and all $x\ge 0$, the constant $\gamma(t,x)$ may be roughly upper bounded by $t^{-1}x^t$. Observe, finally, that when $\alpha=1$ and $r=0$ the bound in Corollary \ref{ctg} is trivial since $\gamma(0,\frac{1}{2})=+\infty$. \\

The next corollary studies the case of an arbitrary $\ell\in{\rm rv}^0_{\infty}$. First, let $\ell\in{\rm rv}^0_{\infty}$ and denote, for all $\beta\in(0,1)$ and all $x\ge 1$, 
\begin{equation}
\label{lb}
\ell^{\,\circ}_{\beta}(x)=\sqrt{\int_{2x}^{+\infty}\frac{\ell(u)^{2}}{u^{2-\beta}}{\rm d}u}.
\end{equation}
Then, one may deduce that $\ell^{\,\circ}_{\beta}\in {\rm rv}^{-(1-\beta)/2}_{\infty}$ and satisfies, 
\begin{equation}
\label{equiv1}
\ell^{\,\circ}_{\beta}(x)\sim\frac{\ell(x)}{(2x)^{\frac{1-\beta}{2}}\sqrt{1-\beta}},
\end{equation}
as $x\to+\infty$, by an application of Karamata's Theorem \citep{K33}. We are now in position to state our next result.

\begin{cor}
\label{ctg2}
Suppose that $\mathcal S$ is infinite. Assume that, for $\alpha\in[0,1]$ and $\ell\in {\rm rv}^0_{\infty}$, we have $\nu(\e)\le \e^{-\alpha}\ell(1/\e)$ for all $0<\e\le1$. Then, for all $n\ge 2$, all $0\le r\le n-1$ and all $\beta\in(0,1)$ with $\beta>2(\alpha-r)-1$, we have
$$\esp M_{n,r}\le c(r)\,n^{\alpha-1}\ell(n)+c_2(\alpha,\beta,r)\,n^{\alpha-\frac{1+\beta}{2}}\ell^{\,\circ}_{\beta}(n),$$
where 
\begin{eqnarray}
c_2(\alpha,\beta,r)&=&\frac{4^{1+r}}{r!}\left(\frac{1+r}{2}\right)^{\frac{1+\beta}{2}+r-\alpha}\sqrt{\gamma(1+\beta+2(r-\alpha),1)},
\nonumber
\end{eqnarray}
$c(r)$ is as in \eqref{cr}, and $\gamma(\cdot,\cdot)$ denotes the lower incomplete Gamma function defined in \eqref{lig}.
\end{cor}

Observe that, for every $\beta\in(0,1)$ with $\beta>2(\alpha-r)-1$, this bound is rate optimal according to \eqref{asympM} since, using \eqref{equiv1}, we have
\begin{equation}
\label{equiv2}
n^{\alpha-\frac{1+\beta}{2}}\ell^{\,\circ}_{\beta}(n)\sim\frac{n^{\alpha-1}\ell(n)}{2^{\frac{1-\beta}{2}}\sqrt{1-\beta}},
\end{equation}
as $n\to+\infty$. Hence, the result in Corollary \ref{ctg2} differs from that of Corollary \ref{ctg} mainly at the level of constants.\\

Next, we present an additional result in the spirit of Theorem \ref{tg}, which will shed an interesting light on the lower bounds presented further. This result is less explicit than Theorem \ref{tg}, but allows for tighter upper bounds in certain cases, including when the counting function is regularly varying with exponent $\alpha\in(0,1]$. First, we introduce the function $\kappa_+$ defined, for all $\e\in(0,1]$, by 
\begin{equation}
\kappa_+(\e)=\sup_{0<u\le\e}\frac{\nu(u/2)}{\nu(u)}.
\label{kplus}
\end{equation}
Note that $\kappa_+$ is non-decreasing by construction. Also, given that $\nu$ is non-increasing, we have $\kappa_+(\e)\ge 1$ for all $0<\e\le1$. For simplicity of notation we write
\begin{equation}
\label{kop}
\kappa_+^0 = \lim_{\e\to0}\kappa_+(\e).
\end{equation}

\begin{theo}
\label{tgplus}
For any $n\ge 1$ and any $0\le r\le n-1$, we have
$$\esp M_{n,r}\le \inf_{0<\e\le1/2}\left\{\varphi^{\,+}_{n,r}(\e)+{\theta}^{\,+}_{n,r}(\e)\right\},$$
where $\varphi^{\,+}_{n,r}(\e)$ is defined in Theorem \ref{tg} and
$${\theta}^{\,+}_{n,r}(\e)=2^{1+r}\binom{n}{r}\int_{0}^{\e}(\kappa_+(2u)-1)\nu(u)u^r\left(1-\frac u2\right)^{n-r}{\rm d}u.$$
\end{theo}

The monotonicity of $\nu$ leads, immediately, to the fact that $(\kappa_+(2u)-1)\nu(u)\le(\kappa_+(2u)-1)\nu(u/2)$, for all $0<u\le1$. As a result, by monotonicity of $\kappa_+$, $\theta^{\,+}_{n,r}(\e)\le\psi^{\,+}_{n,r}(\e)$ for $0<\e\le1/2$ provided $\kappa_{+}(2\e)\le 2$. The next statement shows that when $\nu\in {\rm rv}^{\alpha}_{0}$ this condition is always satisfied for $\e$ small enough leading to a potentially tighter bound than Theorem \ref{tg}.

\begin{pro}
\label{limkplus}
Suppose that $\nu\in {\rm rv}^{\alpha}_{0}$, for $\alpha\in[0,1]$. Then $\kappa_+^0=2^{\alpha}$.
\end{pro}

The proof of Proposition \ref{limkplus} follows, almost immediately, from the definition of slowly varying functions at $+\infty$, and is therefore omitted. We end this subsection by the following corollary.

\begin{cor}
\label{ctg3} Suppose that $\mathcal S$ is infinite. Assume that $\kappa^0_+\in(1,2]$ and that, for $\alpha\in[0,1]$ and $\ell\in {\rm rv}^0_{\infty}$, we have $\nu(\e)\le \e^{-\alpha}\ell(1/\e)$ for all $0<\e\le1$. Then, for all $n\ge 2$ large enough so that
\begin{equation}
\label{condkplus}
\kappa_+\left(\frac2n\right)\le2\kappa_+^0-1,
\end{equation}
all $0\le r\le n-1$ and all $\beta\in(0,1)$ with $\beta>2(\alpha-r)-1$, we have
$$\esp M_{n,r}\le c(r)\,n^{\alpha-1}\ell(n)+(\kappa_+^0-1)c_2(\alpha,\beta,r)\,\left(\frac n2\right)^{\alpha-\frac{1+\beta}{2}}\ell^{\,\circ}_{\beta}\left(\frac n2\right),$$
where $c(r)$ is as in \eqref{cr}, $c_2(\alpha,\beta,r)$ is as given in Corollary \ref{ctg2}, and $\gamma(\cdot,\cdot)$ denotes the lower incomplete Gamma function defined in \eqref{lig}.
\end{cor}

According to Proposition \ref{limkplus}, the assumptions of Corollary \ref{ctg3} are satisfied when $\nu\in{\rm rv}^{\alpha}_{0}$ for $\alpha\in(0,1]$, in which case $(\kappa^0_{+}-1)\in(0,1]$ and therefore $(\kappa_+^0-1)c_2(\alpha,\beta,r)\le c_2(\alpha,\beta,r)$. As in Corollary \ref{ctg2}, note that the bound is rate optimal, for all $\beta\in(0,1)$ with $\beta>2(\alpha-r)-1$, thanks to \eqref{equiv2}. Lastly, observe that, with additional information on $\ell$, the range of $n$ for which \eqref{condkplus} applies can be made explicit.

\subsection{Lower bounds}
In this subsection we tackle the problem of finding non-asymptotic lower bounds for the expectation of the occupancy probabilities. For this purpose, we introduce the function $\kappa_-$ defined, for all $\e\in(0,1]$, by
\begin{equation}
\label{kminus}
\kappa_-(\e) = \sup_{0<u\le\e}\frac{\nu(u)}{\nu(u/2)}.
\end{equation}
Note that $\kappa_{-}$ is non-decreasing and satisfies $\kappa_{-}(\e)\le1$. We further define
\begin{equation}
\label{kom}
\kappa_-^0 = \lim_{\epsilon\to0}\kappa_-(\epsilon).
\end{equation}
The following result is in the spirit of Theorem \ref{tgplus}.

\begin{theo}
\label{low}
For any $n\ge 1$ and any $0\le r\le n-1$, we have
$$\esp M_{n,r}\ge\sup_{0<\e\le1/2}\left\{\varphi^{\,-}_{n,r}(\e)+\theta^{\,-}_{n,r}(\e)\right\},$$
where
\begin{eqnarray}
\varphi^{\,-}_{n,r}(\e)&=& {n\choose r}\nu(\e) \e^{r+1}(1-p_{\star})^{n-r},\nonumber\\
\theta^{\,-}_{n,r}(\e)&=&2^{-r}{n\choose r}\int_{0}^{\e} (1-\kappa_-(2u))\nu(u)u^{r}\left(1-2u\right)^{n-r}
{\rm d}u,
\nonumber
\end{eqnarray}
and $p_{\star} = \max\{p_a:a\in \mathcal A\}$.
\end{theo}

In order for the term $\theta^{\,-}_{n,r}(\e)$ to be strictly positive, there needs to be at least one $\e\in(0,1/2]$ with $\kappa_{-}( 2\e)<1$. The next proposition indicates that this requirement holds when $\nu\in {\rm rv}^{\alpha}_{0}$ with $\alpha\in(0,1]$.

\begin{pro}
\label{limkminus}
Suppose that $\nu\in {\rm rv}^{\alpha}_{0}$, for $\alpha\in[0,1]$.  Then $\kappa_-^0=2^{-\alpha}$.
\end{pro}

The proof of Proposition \ref{limkminus} follows, almost immediately, from the definition of slowly varying functions at $+\infty$ and is thus omitted.\\

Given the monotonicity of $\kappa_{-}$ and the fact that $\kappa_{-}(\e)\le 1$, Proposition \ref{limkminus} implies that, for $\alpha=0$, $\kappa_{-}$ is identically equal to $1$ and the second term of the bound is therefore equal to $0$. The reader may easily check that this last observation also holds when $\mathcal S$ is finite since, obviously, $\kappa_{-}(\e)\to 1$ as $\e\to0$ in this case. Thus, the term $\theta^{\,-}_{n,r}(\e)$ contributes to the bound when $\nu\in {\rm rv}^{\alpha}_{0}$ with $\alpha\in(0,1]$ and is identically $0$ when the support $\mathcal S$ is finite or $\nu\in {\rm rv}^{0}_{0}$. Note, however, that the lower bound is attained for uniform distributions. Indeed, suppose that $2\le\vert\mathcal S\vert <+\infty$ and that $P$ is uniform. Setting $\e_0=1/|\mathcal S|$ we have $\nu(\e_0)=\vert\mathcal S\vert$, $\kappa_{-}(\e_0)=1$, and
\begin{eqnarray*}
\esp M_{n,r} &=& {n\choose r}\sum_{k=1}^{|S|} \left(\frac{1}{|S|}\right)^{r+1}\left(1-\frac{1}{|S|}\right)^{n-r}\\
&=& {n\choose r}\nu(\e_0) \e_0^{r+1}(1-p_{\star})^{n-r} = \sup_{0<\e\le1/2}\varphi^{\,-}_{n,r}(\e).
\end{eqnarray*}
We end this section with a corollary similar in nature to Corollary \ref{ctg3}. First recall that, for any $t>0$ and any $x\ge0$, we have
$$\left(1-\frac{x}{n}\right)^{n}\to e^{-x}\qquad\mbox{and}\qquad\int_{0}^{x}u^{t-1}\left(1-\frac{u}{n}\right)^{n}{\rm d}u\to\gamma(t,x),$$
as $n\to+\infty$, where the second limit follows by dominated convergence.

\begin{cor}
\label{clow}
Suppose that $ \kappa_{-}^0<1$ and that, for $\alpha\in[0,1]$ and $\ell\in{\rm rv}^0_{\infty}$, we have $\nu(\e)\ge \e^{-\alpha}\ell(1/\e)$ for all $0<\e\le1$. Assume, in addition, that $\ell$ is non-decreasing. Fix $r\ge 0$ and let $n_0$ be the smallest $n\ge\max\{2,1+r\}$ satisfying the conditions
\begin{align}
(a)&\quad\kappa_{-}\left(\frac 2n\right)\le\frac{1+\kappa_{-}^0}{2},
\nonumber\\
(b)&\quad\left(1-\frac{r}{n}\right)^{n}\ge \frac{e^{-r}}{2},
\nonumber\\
(c)&\quad\int_{0}^{2}u^{r-\alpha}\left(1-\frac{u}{n}\right)^{n}{\rm d}u\ge \frac{\gamma(1+r-\alpha,2)}{2}.
\nonumber
\end{align} 
Then, for all $n\ge n_0$, we have
$$\esp M_{n,r}\ge \frac{e^{-r}}{2r!}\left[(1-p_{\star})^n+\frac{(1- \kappa_{-}^0)\gamma(1+r-\alpha,2)}{2^{1-\alpha}4^{1+r}}\right] \frac{\ell(n)}{n^{1-\alpha}},$$
where $p_{\star} = \max\{p_a:a\in\mathcal A\}$.
\end{cor}
According to Proposition \ref{limkminus}, all assumptions of Corollary \ref{clow} are satisfied if $\nu(\e)=\e^{-\alpha}\ell(1/\e)$ with $\alpha\in(0,1]$ and $\ell\in{\rm rv}^0_{\infty}$ is non-decreasing. We do not know whether a similar bound holds for arbitrary $\ell$. However, note that, for all $\ell\in{\rm rv}^0_{\infty}$, one may use the fact that, for $\alpha>0$ and $0<\eta<\alpha$, there exists $0<\e_{\eta}<1$ and $C_\eta>0$ such that $\e^{-\alpha}\ell(1/\e)\ge C_\eta\e^{\eta-\alpha}$ for all $0<\e\le\e_{\eta}$. Unfortunately, this approach yields suboptimal rates of convergence. Finally note that, for practical purposes, one can use the crude lower bound $\gamma(t,x)\ge (te^{x})^{-1}x^{t}$.

%-----------------------------------------
%-----------------------------------------

\section{Applications and extensions}
\label{apps}

\subsection{Turing's formula}
\label{sstf}
In many practical applications one needs to estimate the occupancy probabilities,  $M_{n,r}$. Perhaps the most famous estimator of this quantity is Turing's formula, which was introduced by \citet{G53}, where the ideas were primarily credited to Alan M.\ Turing. For this reason the estimator has come to be called Turing's formula or the Good-Turing formula. It is given by
$$T_{n,r}=\frac{(1+r)K_{n,1+r}}{n}.$$
A heuristic justification for Turing's formula may be obtained as follows. Denote $\hat p_a=n^{-1}\xi_n(a)$ the natural estimator of $p_a$, where $\xi_n$  is defined by \eqref{xi}. Then one has
\begin{eqnarray}
T_{n,r}&=&\frac{1+r}{n}\sum_{a\in\mathcal A}\mathbf 1\left\{\hat p_a=\frac{1+r}{n}\right\}
\nonumber\\
&=&\sum_{a\in\mathcal A}\hat p_a\mathbf 1\left\{\hat p_a=\frac{1+r}{n}\right\}
\nonumber\\
&\approx&\sum_{a\in\mathcal A}p_a\mathbf 1\left\{\hat p_a=\frac{r}{n}\right\} = M_{n,r}.
\nonumber
\end{eqnarray}
Many properties of this estimator, including bias, consistency, and asymptotic normality have been studied, see, e.g., \citet{H59, H68, R68, S79, H81, E83, C84, CL91, MS00, GS04, Z05, ZH08, ZZ09, OD10, OD12, GC15}, and the references therein. Noting that $\esp T_{n,r}=\esp M_{n-1,r}$, the bias of Turing's formula is given by
$$
\mathbb E\left[M_{n,r}-T_{n,r}\right] = \esp M_{n,r}-\esp M_{n-1,r}.
$$
Thus the results of this paper provide upper and lower bounds on the bias of Turing's formula.\\ 

Further, they provide bounds for the bias of certain modifications of Turing's formula. In particular, for the important case $r=0$, a class of modified Turing formulas was introduced in \citet{C88} \citep[see also][]{ZH07}. The motivation comes from the fact that for all $s=1,2,\dots,n$
\begin{eqnarray*}
\esp M_{n,0} &=& \sum_{k\ge1} p_k(1-p_k)^n\\
 &=& \sum_{i=1}^s (-1)^{i+1} \sum_{k\ge1} p_k^i(1-p_k)^{n-i} +(-1)^{s} \sum_{k\ge1} p_k^{s+1}(1-p_k)^{n-s}\\
&=& \sum_{i=1}^s (-1)^{i+1} \frac{\esp K_{n,i}}{\binom{n}{i}} +(-1)^{s}\frac{\esp M_{n,s}}{{n\choose s}}.
\end{eqnarray*}
This suggests the family of estimators
$$
T^{(s)}_{n,0} = \sum_{i=1}^s (-1)^{i+1} \frac{K_{n,i}}{\binom{n}{i}}, \ \ s=1,2,\dots,n,
$$
each with bias
\begin{eqnarray*}
\mathbb E[M_{n,0}-T^{(s)}_{n,0}] = (-1)^{s}\frac{\esp M_{n,s}}{{n\choose s}}.
\end{eqnarray*}
Note that $T^{(1)}_{n,0}=T_{n,0}$ is just Turing's formula. In \citet{C88} it was shown that, so long as $p_{\star}<.5$, we have 
\begin{equation}
\nonumber
|\esp[M_{n,0}-T^{(1)}_{n,0}] |\ge |\esp[M_{n,0}-T^{(2)}_{n,0}] |\ge\dots\ge|\esp[M_{n,0}-T^{(n)}_{n,0}]|.
\end{equation} 
Thus, these modifications reduce the bias, and the amount of bias remaining can be bounded using the results of this paper. Note, however, that controlling the bias of Turing's formula can only be of interest if this bias is shown to be of smaller order than the rate of decay of $M_{n,r}$ itself. The following subsection provides insights in this direction.

\subsection{Bounds in probability}
\label{sscr}
A natural application of the bounds provided in this article is to combine them with concentration bounds for the occupancy counts and probabilities in order to derive bounds in probability. State of the art concentration results for $K_{n,r}$ and $M_{n,r}$ may be found in \citet{OD12} and \citet{BBO15}. For instance, defining 
$$K_{n,\bar r}=\sum_{s\ge r}K_{n,s}$$
and setting 
$$v_{n,r}=2\min\{\esp K_{n,\bar r},\max\{r\esp K_{n,r},(1+r)\esp K_{n,1+r}\}\},$$ 
Proposition 3.5 in \citet{BBO15} states that, for all $t\ge 0$,
$$\vert K_{n,r}-\esp K_{n,r}\vert<\sqrt{4v_{n,r}t}+\frac{2t}{3},$$
with probability at least $1-4e^{-t}$. The results of Section \ref{sg} may be applied to deduce explicit lower and upper bounds for $\esp K_{n,r}$, denoted, respectively, by $k^{-}_{n,r}$ and $k^{+}_{n,r}$, as well as an explicit upper bound $v^{+}_{n,r}$ for $v_{n,r}$. Combining these bounds with the above results implies that, for all $t>0$,  
$$\max\left\{0,k^{-}_{n,r}-\sqrt{4v^{+}_{n,r}t}-\frac{2t}{3}\right\}\le K_{n,r}\le k^{+}_{n,r}+\sqrt{4v^{+}_{n,r}t}+\frac{2t}{3},$$
with probability at least $1-4e^{-t}$. For ease of exposition, we avoid explicit formulas in this case. Instead, we present explicit bounds for the missing mass using the results of \citet{MO03}, which states that, for all $t>0$, the inequalities
\begin{equation}
\label{MO03}
M_{n,0}\le \esp M_{n,0}+\sqrt{\frac{t}{n}}\quad\mbox{and}\quad M_{n,0}\ge \esp M_{n,0}-\sqrt{\frac{2t}{ne}}
\end{equation}
each hold with probability at least $1-e^{-t}$. The following upper bound follows immediately from \eqref{MO03} and Corollary \ref{ctg}.

\begin{cor}
\label{cbmm1}
Suppose that $\mathcal S$ is infinite. Assume that, for $\alpha\in[0,1]$ and a non-increasing function $\ell\in{\rm rv}^0_{\infty}$, we have $\nu(\e)\le \e^{-\alpha}\ell(1/\e)$, for all $0<\e\le 1$. Then, for all $n\ge 2$ and all $t>0$,
$$
M_{n,0}\le\left(e^{-1}+4\gamma\left(1-\alpha,\frac{1}{2}\right)\right) \frac{\ell(n)}{n^{1-\alpha}}+\sqrt{\frac{t}{n}},
$$
with probability at least $1-e^{-t}$.
\end{cor}

The reader may deduce a similar result by using Corollary \ref{ctg2} instead of Corollary \ref{ctg}. Similarly, the reader may deduce a lower bound in probability by using Corollary \ref{clow}. 
There is an important case where we can combine Corollaries \ref{ctg} and \ref{clow} to get upper and lower bounds in probability that hold simultaneously. Specifically, assume that $\mathcal A=\{1,2,\dots\}$ and that $P=\{p_k:k\ge1\}$ is such that, for some $\alpha\in(0,1)$,
$$
p_k = \frac{k^{-1/\alpha}}{\zeta(1/\alpha)} , 
$$
where $\zeta(1/\alpha) = \sum_{k\ge1} k^{-1/\alpha}$ is the Riemann zeta function at $1/\alpha$. In this case the counting function is regularly varying and we can get the following.

\begin{cor}
\label{cbmm3}
Let $P$ be as above. If $n\ge \max\{2, 2^{1/\alpha}\zeta(1/\alpha)\}$ is such that
$$\kappa_-\left(\frac{2}{n}\right)\le \frac{2^{\alpha}+1}{2^{\alpha+1}}\quad\mbox{and}\quad\int_{0}^{2}u^{-\alpha}\left(1-\frac{u}{n}\right)^{n}{\rm d}u\ge \frac{\gamma(1-\alpha,2)}{2},$$
then, for all $t>0$, 
$$\prob\left(m^{-}_{n,0}(t,\alpha)\le M_{n,0}\le m^{+}_{n,0}(t,\alpha)\right)\ge 1-2e^{-t},$$
where we have denoted
\begin{align*}
m^{-}_{n,0}(t,\alpha)&=\frac{(2^{\alpha}-1)\gamma(1-\alpha,2)}{32}\frac{\zeta(1/\alpha)^{-\alpha}}{n^{1-\alpha}}-\sqrt{\frac{2t}{ne}}\\
m^{+}_{n,0}(t,\alpha)&= \left(e^{-1}+4\gamma\left(1-\alpha,\frac{1}{2}\right)\right) \frac{\zeta(1/\alpha)^{-\alpha}}{n^{1-\alpha}}+\sqrt{\frac{t}{n}}.
\end{align*} 

\end{cor}

\subsection{Random number of observations}
\label{ssp}
In this subsection, we study extensions of our main results to the case where the number of observations is random and modelled by a Poisson distribution. This case corresponds to the practical situation in which the time, $t$, during which the observations are collected is fixed, but the number of observations is not. For this purpose, let $(n_t)_{t\ge 0}$ be a non-homogeneous Poisson process with intensity function $\lambda:\R_+\ra\R_+$, independent of the observations $(X_i)_{i\ge 1}$. Here, it is understood that $n_t$ stands for the number of observations collected by time $t$. Defining
$$\xi_{t}(a)=\sum_{i=1}^{n_t}\mathbf 1\{X_{i}=a\},$$
the occupancy counts and occupancy probabilities at time $t$ are, respectively, defined, for all $r\ge 0$, by
\begin{equation}
\nonumber
K_r(t)=\sum_{a\in\mathcal A} \mathbf 1 \{\xi_{t}(a)=r\}     \quad\mbox{and}\quad M_{r}(t)=\sum_{a\in\mathcal A} p_a\mathbf 1 \{\xi_{t}(a)=r\}.
\end{equation}
Then, provided 
$$
\Lambda_t=\int_0^{t}\lambda(u){\rm d}u\to+\infty,
$$
as $t\to\infty$, a slight modification of the proof of \eqref{asympM} reveals that, if $\nu(\e)=\e^{-\alpha}\ell(1/\e)$ with $\alpha\in(0,1)$ and $\ell\in{\rm rv}^0_{\infty}$, then, for all $r\ge 1$,
\begin{equation}
\label{asympMt}
\esp M_{r}(t) \sim \frac{\alpha \Gamma(1+r-\alpha)}{r!} \Lambda_t^{\alpha-1}\ell(\Lambda_t),
\end{equation}
as $t\to+\infty$. An analogous result for $K_{r}(t)$ also holds, but, for simplicity, we focus on $M_{r}(t)$. As in the case of a fixed number of observations, this result sets a benchmark for finite sample bounds. Based on the observation that
$$\esp M_{r}(t)=\frac{\Lambda^r_t}{r!}\sum_{a\in\mathcal A} p^{1+r}_{a}e^{-\Lambda_t p_a},$$
the proofs of Theorems \ref{tg}, \ref{tgplus}, and \ref{low} may be easily adapted to this case. Modifying Theorem \ref{tg}, for instance, shows that for any $t\ge 0$ and any $r\ge 0$, we have
\begin{equation}
\label{adaptt}
\esp M_{r}(t)\le\inf_{0\le\e\le1}\left\{\varphi^{\,+}_{r}(t,\e)+\psi^{\,+}_{r}(t,\e)\right\},
\end{equation}
where
\begin{eqnarray}
\varphi^{\,+}_{r}(t,\e)&=&\frac{\bar{c}(r)\nu(\e)}{\Lambda_t},
\nonumber\\
\psi^{\,+}_{r}(t,\e)&=&\frac{2^{1+r}\Lambda^r_t}{r!}\int_{0}^{\e}\nu\left(\frac u2\right)u^{r}e^{-\frac{\Lambda_tu}{2}}{\rm d}u,
\nonumber
\end{eqnarray}
and $\bar{c}(r)=(1+r)^{1+r}/(r!e^{1+r})$ for all integers $r\ge 0$. For the sake of brevity, we avoid explicitly stating the respective analogs of Theorems \ref{tgplus} and \ref{low} involving the functions $\kappa_+$ and $\kappa_-$. As with the results of Section \ref{sg}, these bounds lead to optimal-rate upper and lower bounds in terms of $t$. For instance, following the lines of the proof of Corollary \ref{ctg}, and under the same assumptions on the counting function, considering $\e=\Lambda_t^{-1}$ in \eqref{adaptt} yields
\begin{equation}\label{boundpoiss}
\esp M_{r}(t)\le  \left[ \bar{c}(r) + \frac{4^{1+r}\gamma(1+r-\alpha,\tfrac12)}{r!}\right] \Lambda_t^{\alpha-1}\ell(\Lambda_t),
\end{equation}
for any $r\ge 0$ and any $t>0$ with $\Lambda_t^{-1}\le1$. One may easily deduce bounds in the spirit of Corollaries \ref{ctg2} and \ref{ctg3} under related assumptions on $\nu$.

\subsection{Arbitrary distributions in a metric space}
\label{ssap}
So far in this article, the distribution, $P$, of our observations has been supported on an arbitrary and at most countable alphabet $\mathcal A$. In this subsection we briefly investigate a generalization of the notion of occupancy probabilities to the context of an arbitrary distribution, $P$, on a metric space $E$. \\

Let $(E,d)$ be a metric space and let $P$ be any probability distribution on $E$ equipped with its Borel $\sigma$-field. Suppose that we are given independent and identically distributed $E$-valued random variables $X_1,\dots,X_n$ with common distribution $P$. Since $P$ may not be discrete, a natural analog of the occupancy probabilities $M_{n,r}$ may be defined as follows. First, for $\delta>0$ and $x\in E$, we let
\begin{equation}
\label{xidelta}
\xi^{(\delta)}_{n}(x)=\sum_{i=1}^{n}\mathbf 1\left\{x\in B_{X_i,\delta}\right\},
\end{equation}
where, for $u\in E$, $B_{u,\delta}=\{x\in E: d(x,u)<\delta\}$. In other words, $\xi^{(\delta)}_{n}(x)$ is the numbers of sample points from which $x$ is at a distance strictly less than $\delta$. Now, let $X$ be an $E$-valued random variable independent of the sample and having distribution $P$. For any integer $0\le r\le n$, we set
\begin{equation}
\label{Mdelta}
M^{(\delta)}_{n,r}=\prob\left(\xi^{(\delta)}_n(X)=r\,|\,X_1,\dots,X_n\right)=\int_E\mathbf 1\left\{\xi^{(\delta)}_{n}(x)=r\right\}P({\rm d}x).
\end{equation}
The random variable $M^{(\delta)}_{n,r}$ represents the (conditional) probability that, given the first $n$ observations, the next one will fall into the $\delta$-neighbourhood of exactly $r$ of them. A similar extension of the missing mass was studied in Section 4 of \cite{BK12}. In our context, a slight generalisation of Theorem 8 in \cite{BK12} can be written as follows. For $A\subset E$, we denote $N(A,\delta)$ the $\delta$-covering number of $A$, i.e. the minimal number of balls $B_{u,\delta}$ needed to cover $A$. 
\begin{theo}
\label{bkgen}
For all $x\in E$ and $t>0$, let $\tau_x(t)=1-P(B_{x,t})$ and $N_x(t,\varrho)=N(B_{x,t},\varrho)$. Then, for all $n\ge 1$,
$$\esp M^{(\delta)}_{n,0}\le \inf_{x,t,\varrho}\left\{\tau_{x}(t)+\frac{N_{x}(t, \varrho)}{ne}\right\},$$
where the infimum is taken over all $x\in E$, all $ t>0$ and all $0<\varrho\le \delta/2$.
\end{theo}
This result involves, in an interesting way, the geometry of the support of $P$. Suppose, for instance, that the support $\mathcal S$ of $P$ is totally bounded. Then, as noted by \cite{BK12}, taking $x$ in $\mathcal S$, $t$ larger than the diameter of $\mathcal S$ and $\varrho=\delta/2$, leads to
$$\esp M^{(\delta)}_{n,0}\le\frac{N_x(t,\delta/2)}{ne},$$
which is a natural analog of their result in the discrete case.\\ 

In the sequel, we develop an alternative approach. First, we introduce an analog of the counting measure $\bs\nu$. For all $\delta>0$, let $\bs\nu_{\delta}$ be the measure on $[0,1]$ defined by
\begin{equation}
\label{measurenudelta}
\int_{0}^{1}f(u)\bs\nu_{\delta}({\rm d}u)=\int_{E}\frac{f(P(B_{x,\delta}))}{P(B_{x,\delta})}P({\rm d}x),
\end{equation}
for all measurable $f:[0,1]\to\R_{+}$. Then, denoting $\mathcal L_{\delta}(\e)=\{x\in E:P(B_{x,\delta})\ge\e\}$, we introduce the function $\nu_{\delta}$ defined, for all $\e\in[0,1]$, by
\begin{equation}
\label{nudelta}\nu_{\delta}(\e)=\bs\nu_{\delta}([\e,1])=\int_{\mathcal L_{\delta}(\e)}P(B_{x,\delta})^{-1}P({\rm d}x).
\end{equation}
The function $\nu_{\delta}$ is a natural analog of the counting function $\nu$ defined for discrete probability measures. Indeed, an easy application of the Dominated Convergence Theorem shows that, if $P$ is discrete and if for some $c>0$ the distance between any two points in its support is lower bounded by $c$, then for any $\e\in(0,1]$
\begin{equation}
\label{convnudelta}
\lim_{\delta\to 0}\nu_{\delta}(\e)=\nu(\e).
\end{equation}
The next result is in the spirit of \eqref{asympM}. Using the notation introduced in Subsection \ref{ssp}, we denote 
$$M^{(\delta)}_{r}(t)=M^{(\delta)}_{n_t,r},$$
where $(n_t)$ stands for a non-homogeneous Poisson process with intensity function $\lambda:\R_+\to\R_+$.

\begin{theo}
\label{ththething}
Fix $\delta>0$. Suppose that for some $\alpha\in[0,1]$ and some $\ell\in{\rm rv}^0_{\infty}$, possibly depending on $\delta$, we have $\nu_{\delta}(\e)=\e^{-\alpha}\ell(1/\e)$. Then, for all $r\ge0$,
\begin{equation}
\label{thething}
\esp M^{(\delta)}_{r}(t)\sim \frac{\alpha \Gamma(1+r-\alpha)}{r!} \Lambda^{\alpha-1}_t\ell(\Lambda_t),
\end{equation}
as $t\to+\infty$, provided $\Lambda_t\to+\infty$ as $t\to+\infty$. 
\end{theo}
To keep the proof simple, we present this result in the context where the number of observations, $n$, follows a non-homogeneous Poisson process. However, this proof can be modified to give an analogous result in the case where $n$ is fixed. Finally, denoting  
$$\kappa^{(\delta)}_-(\e)=\sup_{0<u\le\e}\frac{\nu_{\delta}(u)}{\nu_{\delta}(u/2)}\qquad\mbox{and}\qquad \kappa^{(\delta)}_+(\e)=\sup_{0<u\le\e}\frac{\nu_{\delta}(u/2)}{\nu_{\delta}(u)},$$
for $0<\e\le 1$, the reader may easily check that, in the context of this subsection, Theorems \ref{tg}, \ref{tgplus} and \ref{low} hold exactly provided $\nu$, $\kappa_+$, $\kappa_-$ and $p_{\star}$ are replaced, respectively, by 
$$\nu_{\delta},\quad\kappa^{(\delta)}_+,\quad\kappa^{(\delta)}_-\quad\mbox{and}\quad p^{(\delta)}_{\star}=\sup\{P(B_{x,\delta}):x\in E\}.
$$ 

The results of this subsection could find interesting applications in the context of a continuous time stochastic process $X=(X_t)_{0\le t\le T}$, considering $P$ to be the distribution of the whole path $X=(X_t)_{0\le t\le T}$ or of $X_t$ for some $0\le t\le T$. In order to have relevant information on the generalized counting function $\nu_{\delta}$, one needs explicit upper and lower bounds on the probabilities of balls $P(B_{x,\delta}),\, x\in E$. Results in this direction have been widely studied and may be related to large deviations theory and density estimates for stochastic partial differential equations. Finally, an interesting question is whether the work of \citet{BBO15} on concentration inequalities can be adapted to this general case. This is left for future research.

%-----------------------------------------
%-----------------------------------------

\section{Proofs}
\label{proofs}

\begin{proof}[\bf Proof of Theorem \ref{tg}]
For all $n\ge1$ and all $0\le r\le n$,
\begin{equation}
\label{tg:eq1}
\esp  M_{n,r}=\sum_{a\in\mathcal A}p_a\prob\left(\xi_n(a)=r\right).
\end{equation}
For any $a\in\mathcal A$, the variables $\mathbf 1\{X_i=a\}$, $i=1,\dots,n$, are independent and have the same Bernoulli distribution with parameter $p_a$. This implies that  
\begin{equation}
\label{tg:eq2}
\prob\left(\xi_n(a)=r\right)=\binom{n}{r}p^{r}_a(1-p_a)^{n-r}.
\end{equation}
As a result, we deduce from \eqref{tg:eq1} and \eqref{tg:eq2} that, for all $n\ge1$ and all $0\le r\le n$,
\begin{equation}
\label{tg:eq3}
\esp M_{n,r}=\binom{n}{r}\sum_{a\in\mathcal A}p^{r+1}_a(1-p_a)^{n-r}.
\end{equation}
Now, suppose that $0\le\e\le1$, $n\ge1$ and $0\le r\le n-1$ are fixed. Note that, from \eqref{tg:eq3}, we can write
\begin{eqnarray}
\esp M_{n,r}&=&\binom{n}{r}\sum_{a\in\mathcal A}p^{r+1}_{a}(1-p_{a})^{n-r}
\nonumber\\
&=&\binom{n}{r}\sum_{a:p_{a}\ge\e}p^{r+1}_{a}(1-p_{a})^{n-r}+\binom{n}{r}\sum_{a:p_{a}<\e}p^{r+1}_{a}(1-p_{a})^{n-r}
\nonumber\\
&=:&\binom{n}{r}(S_1+S_2).
\label{tg:eq4}
\end{eqnarray}
To bound the first term observe that from the definition of the counting function $\nu$ introduced in \eqref{nu} we obtain
\begin{eqnarray}
\binom{n}{r}S_1&\le&\binom{n}{r}\nu(\e)\sup_{u\in[0,1]}u^{r+1}(1-u)^{n-r}
\nonumber\\
&=&\binom{n}{r}\nu(\e)\frac{(1+r)^{1+r}(n-r)^{n-r}}{(1+n)^{1+n}}.
\label{tg:eq5}
\end{eqnarray}
In the case where $r=0$, the upper bound \eqref{tg:eq5} becomes
\begin{eqnarray}
\binom{n}{r}S_1&\le&\nu(\e)\frac{n^n}{(1+n)^{1+n}}
\nonumber\\
&=&\frac{\nu(\e)}{n}\left(1-\frac{1}{1+n}\right)^{1+n}
\nonumber\\
&\le&\frac{\nu(\e)}{ne},
\label{tg:eq5b}
\end{eqnarray}
where, in \eqref{tg:eq5b}, we have used the fact that $\forall u\in[0,1]:\ (1-u)\le e^{-u}$. In the case $1\le r\le n-1$, developing the binomial coefficient in \eqref{tg:eq5}, we need to evaluate the term
\begin{equation}
\label{tg:eq6}
\frac{n!}{r!(n-r)!}\frac{(1+r)^{1+r}(n-r)^{n-r}}{(1+n)^{1+n}}\,.
\end{equation}
Using the Stirling type bound \citep[see][]{Ro55} 
$$
\sqrt{2\pi}\,n^{n+\frac12} e^{-n+\frac{1}{12n+1}}<n!<\sqrt{2\pi}\,n^{n+\frac12} e^{-n+\frac{1}{12n}},
$$
valid for all $n\ge 1$, we deduce in particular that for all $1\le r\le n-1$, we have the following inequalities 
\begin{align}
n! &< \sqrt{2\pi}\,n^{n+\frac12} e^{-n+\frac{1}{12n}}
\nonumber\\
&\le \sqrt{2\pi}\,n^{n+\frac12} e^{-n+\frac{1}{12(1+r)}}
\nonumber\\
&< \sqrt{2\pi}\,n^{n+\frac12} e^{-n+\frac{1}{12r+1}},
\label{fact1}\\
r! &> \sqrt{2\pi}\,r^{r+\frac12} e^{-r+\frac{1}{12r+1}},
\label{fact2}\\
(n-r)! &> \sqrt{2\pi}\,(n-r)^{n-r+\frac12} e^{-n+r}.
\label{fact3}
\end{align}
Using inequalities \eqref{fact1}, \eqref{fact2} and \eqref{fact3}, we obtain
\[
\frac{n!}{r!(n-r)!} \leq \frac{1}{\sqrt{2\pi}}\frac{  n^{n+\frac12}  }{ r^{r+\frac12} (n-r)^{n-r+\frac12}   },
\]
implying that the expression in \eqref{tg:eq6} can be upper bounded by
\begin{equation}\label{stir1}
\frac{1}{\sqrt{2\pi}} \frac{(1+r)^{1+r}}{r^{r+\frac12}}\frac{n^{n+\frac12}}{ (1+n)^{1+n}  }\frac{1}{(n-r)^{\frac12}}\,.
\end{equation}
Using the inequality $1/(n-r)\le (1+r)/n$, valid for all $0\le r\le n-1$, we obtain
\begin{equation}\label{stir2}
\frac{n^{n+\frac12}}{ (1+n)^{1+n}  }\frac{1}{(n-r)^{\frac12}} \leq (1+r)^{\frac12}\left( \frac{n}{1+n}  \right)^{1+n}\frac1n \leq \frac{(1+r)^{\frac12}}{n}\,.
\end{equation}
Combining \eqref{stir1} and \eqref{stir2}, the term \eqref{tg:eq6} is therefore upper bounded by
\begin{align}
\frac{1}{n\sqrt{2\pi}}\left(1+ \frac{1}{r} \right)^{r+\frac12} (1+r) &\le \frac{1}{n\sqrt{\pi}}\left(1+ \frac{1}{r} \right)^{r} (1+r) 
\nonumber\\
&\le \frac{e(1+r)}{n\sqrt{\pi}}, 
\label{stir3}
\end{align}
where we have used that $(1+1/r)^r\leq e$ for all $r\geq 1$. Combining \eqref{tg:eq5} and \eqref{stir3} brings finally
\begin{equation}
\label{tg:eq9}
\binom{n}{r}S_1\le \frac{e(1+r)}{\sqrt{\pi}}\frac{\nu(\e)}{n},
\end{equation}
for all $1\le r\le n-1$. Combining \eqref{tg:eq5b} and \eqref{tg:eq9}, we have therefore established that 
\begin{equation}
%\label{tg:eq9}
\binom{n}{r}S_1\le\frac{c(r)\nu(\e)}{n},
\end{equation}
for all $0\le r\le n-1$ where $c(r)$ is as in \eqref{cr}. We now focus on bounding the second term in \eqref{tg:eq4}. Toward this end we choose $b>1$ and write 
\begin{align}
S_2&=\sum_{j=0}^{+\infty}\sum_{a:p_a<\e}\mathbf 1\left\{\frac{\e}{b^{j+1}}\le p_{a}<\frac{\e}{b^{j}}\right\}p^{r+1}_{a}(1-p_{a})^{n-r}
\label{tg:eqcut}\\
&\le\sum_{j=0}^{+\infty}\left[\nu\left(\frac{\e}{b^{j+1}}\right)-\nu\left(\frac{\e}{b^j}\right)\right]\left(\frac{\e}{b^{j}}\right)^{r+1}\left(1-\frac{\e}{b^{j+1}}\right)^{n-r}
\nonumber\\
&\le\sum_{j=0}^{+\infty}\nu\left(\frac{\e}{b^{j+1}}\right)\left(\frac{\e}{b^{j}}\right)^{r+1}\left(1-\frac{\e}{b^{j+1}}\right)^{n-r}
\nonumber\\
&=\frac{b}{b-1}\sum_{j=0}^{+\infty}\left(\frac{\e}{b^{j}}-\frac{\e}{b^{j+1}}\right)\nu\left(\frac{\e}{b^{j+1}}\right)\left(\frac{\e}{b^{j}}\right)^{r}\left(1-\frac{\e}{b^{j+1}}\right)^{n-r}
\nonumber\\
&\le\frac{b^{1+r}}{b-1}\sum_{j=0}^{+\infty}\int_{\frac{\e}{b^{j+1}}}^{\frac{\e}{b^{j}}}\nu\left(\frac ub\right)u^{r}\left(1-\frac{u}{b}\right)^{n-r}{\rm d}u
\label{tg:eq10}\\
&=\frac{b^{1+r}}{b-1}\int_{0}^{\e}\nu\left(\frac ub\right)u^{r}\left(1-\frac{u}{b}\right)^{n-r}{\rm d}u.
\label{tg:eq11}
\end{align}
For inequality \eqref{tg:eq10} we have used the fact that the functions $u\mapsto u^r$ and $u\mapsto\nu(u)(1-u)^{n-r}$ are respectively non-decreasing and non-increasing so that, for all $u\in[\e b^{-j-1},\e b^{-j}]$, we have
$$\nu\left(\frac{\e}{b^{j+1}}\right)\left(\frac{\e}{b^{j}}\right)^{r}\left(1-\frac{\e}{b^{j+1}}\right)^{n-r}\le b^r\nu\left(\frac ub\right)u^{r}\left(1-\frac{u}{b}\right)^{n-r}.$$
Hence, equation \eqref{tg:eq11} implies that
$$\binom{n}{r}S_2\le\frac{b^{1+r}}{b-1}\binom{n}{r}\int_{0}^{\e}\nu\left(\frac ub\right)u^{r}\left(1-\frac{u}{b}\right)^{n-r}{\rm d}u,$$
which, along which equation \eqref{tg:eq9} and the choice of $b=2$, proves the first claim in Theorem \ref{tg}. We next turn to the inequality \eqref{requaln}. Again, suppose that $\e\in[0,1]$ is fixed and note that, for $r=n$, \eqref{tg:eq4} becomes
\begin{equation}
\esp M_{n,n}=\sum_{a:p_{a}\ge\e}p^{n+1}_{a}+\sum_{k:p_{a}<\e}p^{n+1}_{a}.
\label{tge14}
\end{equation}
Bounding each $p_a$ by $p_{\star}$ in the first sum and by $\e$ in the second, we obtain
\begin{equation}
\esp M_{n,n}\le p^{n+1}_{\star}\nu(\e)+\e^n\sum_{a:p_{a}\le\e}p_{a}\le p^{n+1}_{\star}\nu(\e)+\e^n,
\label{tge15}
\end{equation}
which completes the proof. 
\end{proof}

\begin{proof}[\bf Proof of Corollary \ref{ctg}]
Let $n\ge 1$ and $0\le r\le n-1$ be fixed. Theorem \ref{tg} implies, in particular, that $\esp M_{n,r}\le\varphi^{\,+}_{n,r}(1/n)+\psi^{\,+}_{n,r}(1/n)$. Given the assumption on the counting function, we have
\begin{equation}
\label{ctg:eq1}
\varphi^{\,+}_{n,r}(1/n) \le \frac{c(r)\ell(n)}{n^{1-\alpha}}.
\end{equation}
To bound the second term, note that
\begin{equation}
\nonumber
\psi^{\,+}_{n,r}(1/n)=2^{1+r}{n\choose r}I_n\quad\mbox{where}\quad I_n=\int_{0}^{\frac{1}{n}}\nu\left(\frac u2\right)u^r\left(1-\frac u2\right)^{n-r}{\rm d}u.
\end{equation}
Since $\nu(\e)\le\e^{-\alpha}\ell(1/\e)$, we deduce that  
\begin{eqnarray}
I_n&\le& 2^{\alpha}\int_{0}^{\frac{1}{n}}\ell\left(\frac 2 u\right)u^{r-\alpha}\left(1-\frac u2\right)^{n-r}{\rm d}u
\nonumber\\
&=& 2^{1+r}\int_0^{\frac{1}{2n}} \ell\left(\frac 1 u\right)u^{r-\alpha}(1-u)^{n-r}{\rm d}u
\label{ctg:eqi}\\
&\le& 2^{1+r}\ell(n)\int_0^{\frac{1}{2n}} u^{r-\alpha}(1-u)^{n-r}{\rm d}u, 
\label{ctg:eqi2}
\end{eqnarray}
where \eqref{ctg:eqi} follows from a change of variables and \eqref{ctg:eqi2} uses the fact that $\ell$ is non-increasing. Then, since $(1-u)\le e^{-u}$ for $0\le u\le 1$, we have
\begin{eqnarray}
I_n&\le& 2^{1+r}\ell(n)\int_0^{\frac{1}{2n}} u^{r-\alpha}e^{-(n-r)u}{\rm d}u 
\nonumber\\
&=& \frac{2^{1+r}\ell(n)}{(n-r)^{1+r-\alpha}}\int_0^{\frac{n-r}{2n}} u^{r-\alpha}e^{-u}{\rm d}u 
\nonumber\\
&\le& \frac{2^{1+r}\ell(n)}{(n-r)^{1+r-\alpha}}\int_0^{\frac{1}{2}} u^{r-\alpha}e^{-u}{\rm d}u 
\nonumber\\
&\le& \frac{2^{1+r}\ell(n)(1+r)^{1+r-\alpha}}{n^{1+r-\alpha}}\int_0^{\frac{1}{2}} u^{r-\alpha}e^{-u}{\rm d}u 
\label{ctg:eq2}\\
&=& \frac{2^{1+r}\ell(n)(1+r)^{1+r-\alpha}}{n^{1+r-\alpha}}\gamma(1+r-\alpha,\tfrac 1 2),
\nonumber
\end{eqnarray}
where, in \eqref{ctg:eq2}, we used the fact that $1/(n-r)\le (1+r)/n$ for $r\le n-1$. Finally, using the fact that ${n\choose r}\le n^r/r!$, we obtain
\begin{equation}
\label{ctg:eq3}
\psi^{\,+}_{n,r}(1/n)\le \frac{4^{1+r}}{r!}(1+r)^{1+r-\alpha}\gamma(1+r-\alpha,\tfrac 1 2) \frac{\ell(n)}{n^{1-\alpha}}. 
\end{equation}
Combining \eqref{ctg:eq1} and \eqref{ctg:eq3} gives the result. 
\end{proof}

\begin{proof}[\bf Proof of Corollary \ref{ctg2}] The proof of Corollary \ref{ctg2} follows along the same lines as the proof of Corollary \ref{ctg} up to \eqref{ctg:eqi}. Then, applying Cauchy-Schwarz's inequality, we obtain for all $\beta\in(0,1)$ such that $\beta>2(\alpha-r)-1$,
\begin{eqnarray}
I_n&\le&2^{1+r}\int_0^{\frac{1}{2n}} \ell\left(\frac 1 u\right)u^{r-\alpha}(1-u)^{n-r}{\rm d}u
\nonumber\\
&=&2^{1+r}\int_0^{\frac{1}{2n}} u^{-\frac{\beta}{2}}\ell\left(\frac 1 u\right)u^{r-\alpha+\frac{\beta}{2}}(1-u)^{n-r}{\rm d}u
\nonumber\\
&\le&2^{1+r}\sqrt{\int_0^{\frac{1}{2n}} u^{-\beta}\ell\left(\frac 1 u\right)^2{\rm d}u}\sqrt{\int_0^{\frac{1}{2n}}u^{2(r-\alpha)+\beta}(1-u)^{2(n-r)}{\rm d}u}
\nonumber\\
&\le&2^{1+r}\ell^{\,\circ}_{\beta}(n)\sqrt{\int_0^{\frac{1}{2n}}u^{2(r-\alpha)+\beta}(1-u)^{2(n-r)}{\rm d}u},
\label{ctg2:eq1}
\end{eqnarray}
where \eqref{ctg2:eq1} follows from \eqref{lb} and a change of variables. Then, from similar arguments as in the proof of Corollary \ref{ctg}, we deduce
\begin{eqnarray}
\psi^{\,+}_{n,r}(1/n)&=&2^{1+r}{n\choose r}I_n
\nonumber\\
&\le&\frac{4^{1+r}}{r!}n^r\ell^{\,\circ}_{\beta}(n)\sqrt{\int_0^{\frac{1}{2n}}u^{2(r-\alpha)+\beta}(1-u)^{2(n-r)}{\rm d}u}
\nonumber\\
 &\le& \frac{4^{1+r}}{r!}n^r\ell^{\,\circ}_{\beta}(n)\sqrt{\int_0^{\frac{1}{2n}}u^{2(r-\alpha)+\beta}e^{-2u(n-r)}{\rm d}u}
\nonumber\\
&=&\frac{4^{1+r}}{r!}n^r\ell^{\,\circ}_{\beta}(n)\sqrt{\int_0^{1-\frac{r}{n}}u^{2(r-\alpha)+\beta}e^{-u}{\rm d}u} } \left(\frac{1}{2(n-r)}\right)^{r-\alpha +\frac{\beta+1}{2}
\nonumber\\
&\le&c_2(\alpha,\beta,r)\,n^{\alpha-\frac{1+\beta}{2}}\ell^{\,\circ}_{\beta}(n),
\nonumber
\end{eqnarray}
which completes the proof. 
\end{proof}

\begin{proof}[\bf Proof of Theorem \ref{tgplus}]
The proof of Theorem \ref{tgplus} follows the same lines as the proof of Theorem \ref{tg} up to \eqref{tg:eqcut}, where we take $b=2$. Then, we write
\begin{align}
S_2&=\sum_{j=0}^{+\infty}\sum_{a:p_a<\e}^{+\infty}\mathbf 1\left\{\frac{\e}{2^{j+1}}\le p_{a}<\frac{\e}{2^{j}}\right\}p^{r+1}_{a}(1-p_{a})^{n-r}
\nonumber\\
&\le\sum_{j=0}^{+\infty}\left[\nu\left(\frac{\e}{2^{j+1}}\right)-\nu\left(\frac{\e}{2^j}\right)\right]\left(\frac{\e}{2^{j}}\right)^{r+1}\left(1-\frac{\e}{2^{j+1}}\right)^{n-r}
\nonumber\\
&=\sum_{j=0}^{+\infty}\left[\frac{\nu\left(\frac{\e}{2^{j+1}}\right)}{\nu\left(\frac{\e}{2^j}\right)}-1\right]\nu\left(\frac{\e}{2^j}\right)\left(\frac{\e}{2^{j}}\right)^{r+1}\left(1-\frac{\e}{2^{j+1}}\right)^{n-r}
\nonumber\\
&\le\sum_{j=0}^{+\infty}\left[\kappa_+\left(\frac{\e}{2^{j}}\right)-1\right]\nu\left(\frac{\e}{2^j}\right)\left(\frac{\e}{2^{j}}\right)^{r+1}\left(1-\frac{\e}{2^{j+1}}\right)^{n-r}
\nonumber\\
&=2\sum_{j=0}^{+\infty}\left(\frac{\e}{2^{j}}-\frac{\e}{2^{j+1}}\right)\left[\kappa_+\left(\frac{\e}{2^{j}}\right)-1\right]\nu\left(\frac{\e}{2^j}\right)\left(\frac{\e}{2^{j}}\right)^{r}\left(1-\frac{\e}{2^{j+1}}\right)^{n-r}
\nonumber\\
&\le2^{1+r}\sum_{j=0}^{+\infty}\int_{\frac{\e}{2^{j+1}}}^{\frac{\e}{2^{j}}}(\kappa_+(2u)-1)\nu(u)u^{r}\left(1-\frac{u}{2}\right)^{n-r}{\rm d}u
\label{tge11}\\
&=2^{1+r}\int_{0}^{\e}(\kappa_+(2u)-1)\nu(u)u^{r}\left(1-\frac{u}{2}\right)^{n-r}{\rm d}u,
\label{tge11b}
\end{align}
where, in \eqref{tge11}, we use the monotonicity of both $u\mapsto (\kappa_+(u)-1)u^r$ and $u\mapsto\nu(u)(1-u)^{n-r}$. This completes the proof. %$\square$
\end{proof}

\begin{proof}[\bf Proof of Corollary \ref{ctg3}] Using the same arguments as in the beginning of the proof of Corollary \ref{ctg}, but $\theta^+_{n,r}(1/n)$ in place of $\psi^+_{n,r}(1/n)$,
we obtain 
\begin{align}
\esp M_{n,r}&\le \frac{c(r)\ell(n)}{n^{1-\alpha}}+2^{1+r}{n\choose r}J_n,
\label{ctg3:e1}
\end{align}
where 
\begin{align}
J_n &= \int_0^{\frac{1}{n}} (\kappa_+(2u)-1)\nu(u)u^{r}\left(1-\frac u2\right)^{n-r}{\rm d}u.
\nonumber
\end{align}
Note that
\begin{align}
J_n &\le \int_0^{\frac{1}{n}} (\kappa_+(2u)-1)\ell\left(\frac 1 u\right)u^{r-\alpha}\left(1-\frac u2\right)^{n-r}{\rm d}u.
\nonumber
\end{align}
Given assumption \eqref{condkplus}, and the fact that $\kappa_+$ is non-decreasing, we deduce that 
\begin{align}
J_n &\le(\kappa_+(2/n)-1) \int_0^{\frac{1}{n}}\ell\left(\frac 1 u\right)u^{r-\alpha}\left(1-\frac u2\right)^{n-r}{\rm d}u
\nonumber\\
&\le2(\kappa^0_+-1) \int_0^{\frac{1}{n}}\ell\left(\frac 1 u\right)u^{r-\alpha}\left(1-\frac u2\right)^{n-r}{\rm d}u.
\label{ctg3:e2}
\end{align}
Proceeding now as in the proof of Corollary \ref{ctg2} and applying Cauchy-Schwarz's inequality in \eqref{ctg3:e2} leads, for all $\beta>2(\alpha-r)-1$, to
\begin{align}
J_n &\le 2(\kappa^0_+-1) \int_0^{\frac{1}{n}}u^{-\frac{\beta}{2}}\ell\left(\frac 1 u\right)u^{r-\alpha+\frac{\beta}{2}}\left(1-\frac u2\right)^{n-r}{\rm d}u
\nonumber\\
&\le2(\kappa^0_+-1) \sqrt{\int_0^{\frac{1}{n}}u^{-\beta}\ell\left(\frac 1 u\right)^2{\rm d}u}\sqrt{\int_0^{\frac{1}{n}}u^{2(r-\alpha)+\beta}\left(1-\frac u2\right)^{2(n-r)}{\rm d}u}
\nonumber\\
&=2(\kappa^0_+-1) \ell^{\circ}_\beta\left(\frac n2\right)\sqrt{\int_0^{\frac{1}{n}}u^{2(r-\alpha)+\beta}\left(1-\frac u2\right)^{2(n-r)}{\rm d}u},
\label{ctg3:e3}
\end{align}
where \eqref{ctg3:e3} follows from \eqref{lb} and a change of variables. Using, as in the proofs of Corollary \ref{ctg} and Corollary \ref{ctg2}, the fact that $(1-u)\le e^{-u}$ for $0\le u\le 1$ and the observation that $1/(n-r)\le (1+r)/n$ for $r\le n-1$, the reader may easily check that the square root term in \eqref{ctg3:e3} is upper bounded by
\begin{equation}
\left(\frac{1+r}{n}\right)^{\frac{1+\beta}{2}+r-\alpha}\sqrt{\gamma(1+\beta+2(r-\alpha),1)}.
\nonumber
\end{equation}
Finally, combining this last observation with \eqref{ctg3:e3}, and the fact that ${n\choose r}\le n^r/r!$, we deduce that the second term on the right hand-side of \eqref{ctg3:e1} is at most
\begin{align}
&2^{2+r}{n\choose r}(\kappa^0_+-1) \ell^{\circ}_\beta\left(\frac n2\right)\left(\frac{1+r}{n}\right)^{\frac{1+\beta}{2}+r-\alpha}\sqrt{\gamma(1+\beta+2(r-\alpha),1)}
\nonumber\\
& \le\frac{2^{2+r}}{r!}(\kappa^0_+-1) \ell^{\circ}_\beta\left(\frac n2\right)n^{\alpha-\frac{1+\beta}{2}}(1+r)^{\frac{1+\beta}{2}+r-\alpha}\sqrt{\gamma(1+\beta+2(r-\alpha),1)}
\nonumber\\
& = (\kappa_+^0-1)c_2(\alpha,\beta,r)\left(\frac n2\right)^{\alpha-\frac{1+\beta}{2}}\ell^{\circ}_\beta\left(\frac n2\right),
\nonumber
\end{align}
where $c_2(\alpha,\beta,r)$ is as in Corollary \ref{ctg2}. This completes the proof.
\end{proof}

\begin{proof}[\bf Proof of Theorem \ref{low}]
Fix $n\ge 1$, $0\le r\le n-1$ and $\e\in(0,1/2]$. As in the proof of Theorem \ref{tg}, we write 
\begin{eqnarray*}
\mathbb E M_{n,r} &=& {n\choose r}\sum_{a:p_a\ge\e} p_a^{r+1}(1-p_a)^{n-r} +{n\choose r}\sum_{a:p_a<\e} p_a^{r+1}(1-p_a)^{n-r} \\
&=:& {n\choose r}\left(S_1 + S_2\right).
\end{eqnarray*}
From the definition of $\nu$, it is clear that  
$$
S_1 \ge\nu(\e) \e^{r+1}(1-p_{\star})^{n-r}.
$$
To bound $S_2$, we write
\begin{align}
S_2 &= \sum_{j=0}^\infty \sum_{a:p_a<\e} \mathbf 1\left\{\frac{\e}{2^{j+1}}\le p_a<\frac{\e}{2^j}\right\}p_a^{r+1}(1-p_a)^{n-r}
\nonumber\\
&\ge \sum_{j=0}^\infty \left[\nu\left(\frac{\e}{2^{j+1}}\right)-\nu\left(\frac{\e}{2^j}\right)\right]\left(\frac{\e}{2^{j+1}}\right)^{r+1}\left(1-\frac{\e}{2^{j}}\right)^{n-r}
\nonumber\\
&= \sum_{j=0}^\infty\left(\frac{\e}{2^j}-\frac{\e}{2^{j+1}}\right)\left[1-\frac{\nu\left(\frac{\e}{2^{j}}\right)}{\nu\left(\frac{\e}{2^{j+1}}\right)}\right] \nu\left(\frac{\e}{2^{j+1}}\right)\left(\frac{\e}{2^{j+1}}\right)^{r}\left(1-\frac{\e}{2^{j}}\right)^{n-r}
\nonumber\\
&\ge \sum_{j=0}^\infty\left(\frac{\e}{2^j}-\frac{\e}{2^{j+1}}\right)\left[1-\kappa_-\left(\frac{\e}{2^{j}}\right)\right] \nu\left(\frac{\e}{2^{j+1}}\right)\left(\frac{\e}{2^{j+1}}\right)^{r}\left(1-\frac{\e}{2^{j}}\right)^{n-r}
\nonumber\\
&\ge 2^{-r}\sum_{j=0}^\infty \int_{\frac{\e}{2^{j+1}}}^{\frac{\e}{2^j}} (1-\kappa_-(2u))\nu(u)u^{r}\left(1-2u\right)^{n-r}
{\rm d}u
\label{lowertg:eq1}\\
&= 2^{-r} \int_{0}^{\e} (1-\kappa_-(2u))\nu(u)u^{r}\left(1-2u\right)^{n-r}
{\rm d}u
\nonumber
\end{align}
where in \eqref{lowertg:eq1}, we used the monotonicity of both $u \mapsto u^{r}$ and $u \mapsto (1-\kappa_-(2u))\nu(u)\left(1-2u\right)^{n-r}$. This completes the proof. 
\end{proof}

\begin{proof}[\bf Proof of Corollary \ref{clow}]
Fix $n\ge n_0$. The bound in Corollary \ref{clow} is obtained by taking $\e=n^{-1}$ in Theorem \ref{low}. First, using the assumption on $\nu$, note that 
\begin{eqnarray}
\varphi^-_{n,r}(1/n)
&\ge& \frac{{n\choose r}}{n^{r}} (1-p_{\star})^{n}\frac{\ell(n)}{n^{1-\alpha}}
\nonumber\\
&\ge& \frac{1}{r!}\left(1-\frac{r}{n}\right)^{n}(1-p_{\star})^{n}\frac{\ell(n)}{n^{1-\alpha}}
\label{clow:eq1}\\
&\ge& \frac{e^{-r}}{2r!}(1-p_{\star})^{n}\frac{\ell(n)}{n^{1-\alpha}},
\label{clow:eq2}
\end{eqnarray}
where \eqref{clow:eq1} is due to the fact that
\begin{equation}
\label{clow:eq3}
\frac{{n\choose r}}{n^{r}}\ge\frac{1}{r!}\left(1-\frac{r}{n}\right)^{r}\ge\frac{1}{r!}\left(1-\frac{r}{n}\right)^{n},
\end{equation}
and \eqref{clow:eq2} follows from condition $(b)$. Next, denote
$$J_{n}=\int_{0}^{\frac{1}{n}}(1-\kappa_{-}(2u))\nu(u)u^{r}(1-2u)^{n-r}{\rm d}u.$$
Using condition $(a)$ and the assumption on $\nu$, it may be easily checked that 
\begin{eqnarray}
J_{n}&\ge&\frac{(1-\kappa^0_{-})\ell(n)}{2}\int_{0}^{\frac{1}{n}}u^{r-\alpha}(1-2u)^{n-r}{\rm d}u
\nonumber\\
&=&\frac{(1-\kappa^0_{-})\ell(n)}{2(2n)^{1+r-\alpha}}\int_{0}^{2}u^{r-\alpha}\left(1-\frac un\right)^{n-r}{\rm d}u
\nonumber\\
&\ge&\frac{(1-\kappa^0_{-})\ell(n)}{2(2n)^{1+r-\alpha}}\int_{0}^{2}u^{r-\alpha}\left(1-\frac un\right)^{n}{\rm d}u
\nonumber\\
&\ge&\frac{(1-\kappa^0_{-})\ell(n)}{4(2n)^{1+r-\alpha}}\gamma(1+r-\alpha,2)
\label{clow:eq4}
\end{eqnarray}
where \eqref{clow:eq4} follows from condition $(c)$. By rearranging the terms and using \eqref{clow:eq3} once again along with condition $(b)$, we finally deduce that 
\begin{equation}
\label{clow:eq5}
\theta^-_{n,r}(1/n)= 2^{-r}{n\choose r}J_{n}\ge \frac{e^{-r}}{2r!}(1-\kappa^0_{-})\frac{\gamma(1+r-\alpha,2)}{2^{1-\alpha}4^{1+r}}\frac{\ell(n)}{n^{1-\alpha}}.
\end{equation}
The result follows from \eqref{clow:eq2} and \eqref{clow:eq5}. 
\end{proof}

\begin{proof}[\bf Proof of Corollary \ref{cbmm3}] Let us denote $z=\zeta(1/\alpha)$ for brevity. First, it may be easily verified that the counting function $\nu$ of the distribution considered satisfies, for all $0<\e<1$, $\nu(\e)= \lfloor(z\e)^{-\alpha}\rfloor$, where $\lfloor \cdot\rfloor$ is the floor function. As a result, $\e^{\alpha}\nu(\e)\to z^{-\alpha}$ as $\e\to 0$, and thus $\nu\in {\rm rv}^{\alpha}_{0}$. From here, Proposition \ref{limkminus} implies that $\kappa^0_-=2^{-\alpha}$. Noticing that $p_{\star}=z^{-1}$ and that, for $\e\le (2^{1/\alpha}z)^{-1}$, we have $\nu(\e)\ge(z\e)^{-\alpha}-1\ge (z\e)^{-\alpha}/2$ (i.e. for $x\ge 2^{1/\alpha}z$ we can take $\ell(x)=z^{-\alpha}/2$), Corollary \ref{clow} implies that, provided $n\ge \max\{2,2^{1/\alpha}z\}$ and the conditions
$$\kappa_-\left(\frac{2}{n}\right)\le \frac{2^{\alpha}+1}{2^{\alpha+1}}\quad\mbox{and}\quad\int_{0}^{2}u^{-\alpha}\left(1-\frac{u}{n}\right)^{n}{\rm d}u\ge \frac{\gamma(1-\alpha,2)}{2}
$$
are satisfied (since $r=0$, condition $(b)$ in Corollary \ref{clow} automatically holds), we obtain
\begin{align}
\esp M_{n,0}&\ge \frac{1}{4}\left[(1-z^{-1})^n+\frac{(1-2^{-\alpha})\gamma(1-\alpha,2)}{2^{1-\alpha}4}\right]\frac{z^{-\alpha}}{n^{1-\alpha}}
\nonumber\\
&= \frac{1}{4}\left[(1-z^{-1})^n+\frac{(2^{\alpha}-1)\gamma(1-\alpha,2)}{8}\right]\frac{z^{-\alpha}}{n^{1-\alpha}}
\nonumber\\
&\ge \frac{(2^{\alpha}-1)\gamma(1-\alpha,2)}{32}\frac{z^{-\alpha}}{n^{1-\alpha}},\nonumber
\end{align}
where the last inequality follows from the fact that $z>1$. Now note that $\nu(\e)= \lfloor(z\e)^{-\alpha}\rfloor\le (z\e)^{-\alpha}$. Thus applying Corollary \ref{ctg} with $\ell$ constant and equal to $z^{-\alpha}$ gives an upper bound on $\esp M_{n,0}$. Combining the upper and lower bounds with the concentration bound in \eqref{MO03} yields the desired result.
\end{proof}

\begin{proof}[\bf Proof of Theorem \ref{bkgen}] 
Fix $x\in E$, $t>0$ and $0<\varrho\le\delta/2$. Then, observe that, 
\begin{align}
 \esp M^{(\delta)}_{n,0} &=\int_E \prob\left(\xi^{(\delta)}_{n}(u)=0\right)P({\rm d}u)
\nonumber\\
&=\int_E (1-P(B_{u,\delta}))^nP({\rm d}u)
\label{bkgen:e1}\\
&\le\tau_x(t)+\int_{B_{x,t}} (1-P(B_{u,\delta}))^nP({\rm d}u).
\label{bkgen:e2}
\end{align}
Here, \eqref{bkgen:e1} follows from the fact that, for all $u\in E$, $\xi^{(\delta)}_{n}(u)$ has a Binomial distribution with parameters $n$ and $P(B_{u,\delta})$, and \eqref{bkgen:e2} follows from the fact that $(1-P(B_{u,\delta}))^n\le1$. Next, let $N=N(B_{x,t},\varrho)$ and let $B_1,\dots,B_N\subset E$ be balls with radius $\varrho$ satisfying $B_{x,t}\subset B_1\cup\cdots\cup B_N$. Then, we obtain
\begin{align}
\int_{B_{x,t}} (1-P(B_{u,\delta}))^nP({\rm d}u)&\le \sum_{i=1}^N\int_{B_i} (1-P(B_{u,\delta}))^nP({\rm d}u)
\nonumber\\
&\le \sum_{i=1}^NP(B_i)(1-P(B_i))^n,
\label{bkgen:e3}
\end{align}
where \eqref{bkgen:e3} follows from the fact that, since $\varrho\le\delta/2$, if $u\in B_i$ then necessarily $B_i\subset B_{u,\delta}$. Now exactly as in \eqref{tg:eq5b}, one may deduce that
\begin{align}
\sum_{i=1}^NP(B_i)(1-P(B_i))^n&\le N\sup_{0\le p\le 1}p(1-p)^n
\nonumber\\
&\le\frac{N}{ne}.
\label{bkgen:e4}
\end{align}
The result follows by combining \eqref{bkgen:e2}, \eqref{bkgen:e3}, \eqref{bkgen:e4} and taking the infimum over $x\in E$, $t>0$ and $\varrho\le\delta/2$. 
\end{proof}

\begin{proof}[\bf Proof of Theorem \ref{ththething}]
First, note that 
\begin{equation}
\label{tt:e1}
\esp M^{(\delta)}_{r}(t)=\int_E \prob\left(\xi^{(\delta)}_{n_t}(x)=r\right)P({\rm d}x).
\end{equation}
For all $x\in E$, the variable $\xi^{(\delta)}_{n_t}(x)$ follows a Poisson distribution with parameter $\Lambda_tP(B_{x,\delta})$. Combining this with \eqref{measurenudelta} gives
\begin{eqnarray}
\esp M^{(\delta)}_{r}(t)&=&\frac{\Lambda^r_t}{r!}\int_E P(B_{x,\delta})^re^{-\Lambda_tP(B_{x,\delta})}
P({\rm d}x)
\nonumber\\
&=&\frac{\Lambda^r_t}{r!}\int_{0}^{1}u^{1+r}e^{-\Lambda_tu}\bs\nu_{\delta}({\rm d}u)
\nonumber\\
&=&\frac{\Lambda^r_t}{r!}\mathfrak L^{(\delta)}_{1+r}(\Lambda_t),
\label{tt:e2}
\end{eqnarray}
where $\mathfrak L^{(\delta)}_{1+r}(\cdot)$ stands for the Laplace transform of the measure $\bs\nu^{1+r}_{\delta}({\rm d}u)=u^{1+r}\bs\nu_{\delta}({\rm d}u)$. Now according to equality \eqref{ipp1} of Proposition \ref{ipp} from Appendix \ref{nuprop}, we know that for all $0<\e<1$, 
\begin{equation}
\label{tt:e4}
\bs\nu^{1+r}_{\delta}([0,\e])=-\e^{1+r}\nu_{\delta}(\e)+(1+r)\int_{0}^{\e}u^r\nu_{\delta}(u){\rm d}u.
\end{equation}
The assumption on the function $\nu_{\delta}$ implies that 
\begin{equation}
\label{tt:e5}
 \e^{1+r}\nu_{\delta}(\e)=\e^{1+r-\alpha}\ell(1/\e)
\end{equation}
and
\begin{equation}
\label{tt:e6}
\int_{0}^{\e}u^r\nu_{\delta}(u){\rm d}u= \int_{0}^{\e}u^{r-\alpha}\ell(1/u){\rm d}u  \sim\frac{\e^{1+r-\alpha}\ell(1/\e)}{1+r-\alpha},
\end{equation}
as $\e\to0$, where the equivalent follows from Karamata's Theorem \citep{K33}. Combining \eqref{tt:e4}, \eqref{tt:e5} and \eqref{tt:e6} leads to 
\begin{equation}
\bs\nu^{1+r}_{\delta}([0,\e])\sim \frac{\alpha}{1+r-\alpha}\e^{1+r-\alpha}\ell(1/\e), 
\end{equation}
as $\e\to0$. Finally, applying the Tauberian Theorem \citep[see, e.g., Theorem 2, Section 5, Chapter 13 in][]{F71} and using the fact that $\Gamma(2+r-\alpha)=(1+r-\alpha)\Gamma(1+r-\alpha)$, we deduce that 
$$
\mathfrak L^{(\delta)}_{1+r}(t)\sim \alpha\Gamma(1+r-\alpha)t^{-(1+r-\alpha)}\ell(t), 
$$
as $t\to+\infty$. From here, the result follows from the fact that $\Lambda_t\to+\infty$ as $t\to+\infty$ and identity \eqref{tt:e2}. 
\end{proof}
%-----------------------------------------
%-----------------------------------------

\appendix
\section{Basic properties of the counting function}
\label{nuprop}
The counting function $\nu$, defined in \eqref{nu}, is non-increasing by definition.  As $\e$ tends to $0$, $\nu(\e)$ increases towards $\vert \mathcal S\vert$, the cardinality of the support of $P$, which may, of course, be infinite. Since the masses $p_a$ sum to $1$, it may be easily observed that, for all $0<\e\le 1$, 
$$\nu(\e)\le\e^{-1}.$$
Next, we recall general integration by parts formulas from which we will deduce additional properties of $\nu$. 
\begin{pro}
\label{ipp}
Let $\mu$ be any positive measure on $[0,1]$. Then, for all $\tau\ge1$ and all $0<\e<1$, we have the two identities
\begin{eqnarray}
\int_{[0,\e]}x^{\tau}\mu({\rm d}x)&=& -\e^{\tau}\mu([\e,1])+{\tau}\int_{[0,\e]}x^{\tau-1}\mu([x,1]){\rm d}x,
\label{ipp1}\\
\int_{[\e,1]}x^{\tau}\mu({\rm d}x)&=& +\e^{\tau}\mu([\e,1])+{\tau}\int_{[\e,1]}x^{\tau-1}\mu([x,1]){\rm d}x.
\label{ipp2}
\end{eqnarray}
\end{pro}
The proof is a standard application of Fubini's Theorem and is thus omitted.  Note that, in the above, we did not assume finiteness of the integrals. When an integral on one side is infinite, the above should be interpreted to mean that the integral on the other side is infinite as well. We now reproduce an argument presented at the end of Section 3 in \citet{GHP07}. Letting $\bs\nu$ be the counting measure defined in \eqref{measurenu}, it may be easily seen that, for all $\tau\ge 1$ and all $0<\e<1$,  by \eqref{ipp2} we have
\begin{equation*}
\e^{\tau}\nu(\e)+\tau\int_{[\e,1]}x^{\tau-1}\nu(x){\rm d}x=\int_{[\e,1]}x^{\tau}\bs\nu({\rm d}x)=\sum_{a:p_a\ge\e}p^{\tau}_a\le1.
\end{equation*}
Taking the limit as $\e\to0$ and applying dominated convergence gives
\begin{equation}
\label{ipp21}
\lim_{\e\to 0}\left(\e^{\tau}\nu(\e)+\tau\int_{[\e,1]}x^{\tau-1}\nu(x){\rm d}x\right)=\lim_{\e\to 0}\sum_{a:p_a\ge\e}p^{\tau}_a=\sum_ap^{\tau}_a\le1,
\end{equation}
with equality holding if and only if $\tau=1$. Monotonicity guarantees the convergence of the integral in \eqref{ipp21}, which, together with the result of \eqref{ipp21}, implies that $\e^{\tau}\nu(\e)$ has a limit as $\e\to0$. Finally, if this limit was $c>0$, this would contradict the convergence of the integral in \eqref{ipp21} since we would have $x^{\tau-1}\nu(x)\sim c/x$ as $x\to 0$, which, in turn, would imply the integrability of $1/x$ at 0. Taking $\tau=1$ gives the following.

\begin{cor}
$$\int_{0}^{1}\nu(x){\rm d}x=1\quad\mbox{and}\quad\lim_{\e\to 0}\e\nu(\e)=0.$$
\end{cor}

%-----------------------------------------
%-----------------------------------------

\section{On the accrual function}
\label{appendixod10}
In this appendix, we briefly discuss some properties of the accrual function $F$, which is defined for all $0\le\e\le1$ by
$$F(\e)=\int_{ [0,\e]}x\,\bs\nu({\rm d}x),$$
and of the bound \eqref{od10} provided by \citet{OD10}. First, note that \eqref{ipp1} establishes that the counting and accrual functions are related through the formula 
\begin{equation}
\label{linkfnu}
F(\e)=-\e\nu(\e)+\int_{[0,\e]}\nu(x){\rm d}x,
\end{equation}
for all $0<\e<1$. The next result shows that a straight-forward application of \eqref{od10} provides, at least in the pure power setting, rate optimal bounds up to a log term. Recall that, in the regularly varying setting, rate optimal bounds are ones that, asymptotically, behave as in \eqref{asympM}.

\begin{pro}
\label{appendixod1o}
Suppose that, for some constants $0<C_-<C_+<+\infty$ and some $\alpha\in(0,1)$, the counting function $\nu$ satisfies $C_-\e^{-\alpha}\le\nu(\e)\le C_+\e^{-\alpha}$, for all $0<\e<1$. Then, for all $n\ge 1$, the expected missing mass satisfies
$$\left(1-\frac{1}{n}\right)^n\frac{C^{-}_{\alpha}}{n^{1-\alpha}}\le\esp M_{n,0}\le \frac{1+C^{+}_{\alpha}(\log n)^{1-\alpha}}{n^{1-\alpha}},$$
where $C^{-}_{\alpha}=\max\{0,C_-/(1-\alpha)-C_+\}$ and $C^{+}_{\alpha}=(1-\alpha)^{1-\alpha}\{C_+/(1-\alpha)-C_-\}$.
\end{pro}
The lower bound follows easily from \eqref{linkfnu}, the bounds on $\nu$, and the choice $\e=1/n$ in the lower bound of \eqref{od10}.  Similarly, the upper bound follows easily from \eqref{linkfnu}, the fact that $(1-\e)^n\le e^{-n\e}$, the bounds on $\nu$, and the choice $\e=(1-\alpha)(\log n)/n$ in the upper bound of \eqref{od10}. Supposing that the constant $C^-_{\alpha}>0$, it follows that, since $(1-n^{-1})^n\ge e^{-1}/2$ for large enough $n$, the lower bound in Proposition \ref{appendixod10} is rate optimal. 

\section{On lower bounds}
\label{AppC}
This Appendix gives an interesting and known result, versions of which can be found in e.g.\ Lemma 4.1 of \cite{ABS2000} or Lemma 1 of \cite{Zhang2016}.  

\begin{theo}
Suppose that $|\mathcal S|=\infty$. Then there exists a sequence $(k_n)$ of positive integers, with $k_n\to\infty$ as $n\to\infty$, such that, for any $r\ge0$, 
$$
\liminf_{n\to\infty} \frac{k^{r+1}_n}{{n\choose r}}\esp M_{n,r}\ge e^{-1}
$$
and
$$
\liminf_{n\to\infty} \frac{k^{r+1}_n}{{k_n\choose r}}\esp M_{k_n,r}\ge e^{-1}.
$$
\end{theo}
\begin{proof}[\bf Proof]
We only prove the first inequality as the proof of the second is similar. Let $a_n$ be an element of $\mathcal A$ with $n\le 1/p_{a_n}$ and let $k_n = \lfloor1/p_{a_n}\rfloor$, where $\lfloor\cdot\rfloor$ denotes the floor function. Note that $n\le k_n$,  $(1-p_{a_n})\le k_np_{a_n}\le1$, and $p_{a_n}\to0$ as $n\to\infty$.  For $n>r$ we therefore have 
\begin{align*}
\frac{k^{r+1}_n}{{n\choose r}}\esp M_{n,r}&\ge k_n^{r+1} p_{a_n}^{r+1} (1-p_{a_n})^{n-r} \\
&\ge (1-p_{a_n})^{n+1} \\
&\ge (1-p_{a_n}) \left(1-\frac{1}{k_n}\right)^{n} \\
&\ge (1-p_{a_n}) \left(1-\frac{1}{k_n}\right)^{k_n}.
\end{align*}
The result follows by observing that, for a fixed $r$, the term on the right hand-side of the last inequality tends to $e^{-1}$ as $n\to\infty$.
\end{proof}

%-----------------------------------------
%-----------------------------------------

\section*{Acknowledgements}

The authors wish to thank the Editor and the two anonymous referees whose comments led to improvements in the presentation of this paper. In particular, we thank them for showing us a nicer form for $c(r)$.

\bibliography{bibDGP16}
\end{document}